\documentclass[12pt]{amsart}
\usepackage{amsmath,amsfonts,amsthm,amssymb,mathrsfs,
amsxtra,amscd,latexsym}
\usepackage[hmargin=2cm,vmargin=3cm]{geometry}
 \newtheorem{thm}{Theorem}[section]

 \newtheorem{lem}[thm]{Lemma}
 \newtheorem{prop}[thm]{Proposition}
 
 \newtheorem{dfn}[thm]{Definition}

 \newtheorem{ex}[thm]{Example}
 \theoremstyle{definition}
 \theoremstyle{remark}
 \numberwithin{equation}{section}

\newcommand{\sm}{\left(\begin{smallmatrix}}
\newcommand{\esm}{\end{smallmatrix}\right)}
\newcommand{\mat}{\left(\begin{matrix}}
\newcommand{\emat}{\end{matrix}\right)}
\renewcommand{\(}{\left(}
\renewcommand{\)}{\right)}

\newcommand{\mb}{\mathbb}

\newcommand{\mc}{\mathcal}

\def\CC{\mathbb{C}}
\def\HH{\mathbb{H}}
\def\NN{\mathbb{N}}

\def\ZZ{\mathbb{Z}}

\def\SL{\mathrm{SL}}

\begin{document}

\title{Eichler integrals and harmonic weak Maass forms}


 \author{Dohoon Choi}
 \author{Byungchan Kim}
\author{Subong Lim}

 \address{School of liberal arts and sciences, Korea Aerospace University, 200-1, Hwajeon-dong, Goyang, Gyeonggi 412-791, Republic of Korea}
  \email{choija@kau.ac.kr}	
  
\address{School of Liberal Arts and Institute of Convergence Fundamental Studies, Seoul National University of Science and Technology, 172 Gongreung 2 dong, Nowongu, Seoul, 139-743, Republic of Korea}
\email{bkim4@seoultech.ac.kr}

 \address{School of Mathematics, Korea Institute for Advanced Study, Hoegiro 85, Dongdaemun-gu, Seoul 130-722, Republic of Korea}
  \email{subong@kias.re.kr}



 \thanks{
Dohoon Choi was supported by Basic Science Research Program through the National Research Foundation of Korea (NRF) funded by the Ministry of Education, Science and Technology (NRF2010-0022180).
Byungchan Kim was supported by Basic Science Research Program through the National Research Foundation of Korea (NRF) funded by the Ministry of Education, Science and Technology (NRF2011-0009199).}

\begin{abstract}
 Recently, K. Bringmann, P. Guerzhoy, Z. Kent and K. Ono studied the connection between Eichler integrals and the holomorphic parts of harmonic weak Maass forms on the full modular group.
 In this article, we extend  their result to
 more general  groups, namely, $H$-groups by employing the  theory of supplementary functions introduced and developed 
 by M. I. Knopp and S. Y. Husseini. In particular, we show that the set of 
 Eichler integrals, which have polynomial period functions, is
 the same as the set of holomorphic parts of harmonic weak Maass
 forms of which the non-holomorphic parts are certain period integrals
 of cusp forms. From this we deduce relations among period functions for harmonic weak Maass forms.
\end{abstract}

\maketitle

\section{Introduction} \label{section1}
In a seminal paper \cite{Eichler}, M. Eichler introduced what is now called an Eichler integral for the cusp form $f(z)$ of weight  $k\in2\NN$, which is essentially
\begin{equation} \label{Eichlercusp1}
\int_{z}^{ i \infty} f(\tau) (z-\tau)^{k-2} d \tau.
\end{equation}
This simply defined integral is deeply related to many  areas of mathematics. In particular, G. Shimura \cite{Shimura} and Y. Manin \cite{Manin} showed that these are connected with the  special values of $L$-functions.  On the other hand,
 a number of authors have studied the arithmetic
 properties of harmonic weak Maass forms. Especially, the holomorphic parts of harmonic weak Maass forms encode many beautiful mathematical objects. In particular, S. Zwegers \cite{Zwegers}, K. Bringmann and K. Ono \cite{BO}, and D. Zagier \cite{Zagier} revealed the deep relations between harmonic weak Maass forms and Ramanujan's mock theta functions. For detailed implications and applications, we recommend Ono's nice survey paper \cite{Onosurvey}.

 In this light, it is natural to ask whether we can think of  Eichler integrals for
other modular forms.  Since \eqref{Eichlercusp1} does not converge for non-cusp forms, one has to find another way to define the Eichler integral. In a recent paper \cite{BGK}, K. Bringmann,  P. Guerzhoy, Z. Kent and K. Ono
 encountered this difficulty for modular forms on the full modular group, where they defined
 the Eichler integral as a formal power series. Their  definition is motivated by the following well known fact: for the cusp form $f(z) = \sum_{n=1}^{\infty} a(n) q^n$, the integral \eqref{Eichlercusp1} is essentially
\begin{equation}\label{Eichlercusp2}
\sum_{n=1}^{\infty} a(n) n^{1-k} q^n ,
\end{equation}
where $q = \exp(2\pi i z)$, $\rm{Im}\; z >0$. Once we 
 connect Eichler integrals and the holomorphic parts of harmonic weak Maass forms
 on the full modular group, it is natural to seek a further extension to more general  groups. This is the goal of our paper. The main ideas of \cite{BGK} employ Bol's identity and a functional equation for the associated $L$-series. These ideas beautifully manifested a connection between the period polynomial and the obstruction to modularity, but it is hard to find the contribution from the constant term of  the Eisenstein series as they remarked. To resolve this issue, we use the theory of supplementary functions
 introduced by M. I. Knopp \cite{Kon} in 1962  and developed by M. I. Knopp and S. Y. Husseini \cite{Hus} to study Eichler integrals and Eichler cohomology for Fuchsian groups. 
This idea  enables us not only to obtain a desirable generalization, but also to obtain more precise information for the contribution from the constant terms of harmonic weak Maass forms to period polynomials. Moreover, from this point of view, we can show that the set of  Eichler integrals, which have polynomial period functions, is
 the same as the set of holomorphic parts of harmonic weak Maass
 forms of which the non-holomorphic parts are certain period integrals
 of cusp forms.  Zagier called these cusp forms  {\it shadows} of mock modular forms (for the definitions, consult \cite{Zagier}), and actually the supplementary functions are closely related with the shadows. In this sense, our approach is based on the role of shadows to the theory of harmonic weak Maass forms.

Let $\Gamma$ be an $H$-group, i.e., a finitely generated Fuchsian group of the first kind, which has at least one parabolic class.
 This implies that $\Gamma$ satisfies the following properties (see \cite{Kon2} or \cite{Leh}):
\begin{enumerate}
\item $\Gamma$ is finitely generated,
\item $\Gamma$ is discrete, but discontinuous at no point of the real line,
\item $\Gamma$ contains translations. 
\end{enumerate}
\noindent Let $k\in\ZZ$ and $\chi$ be a (unitary) character of $\Gamma$. A harmonic weak Maass form of weight $k$ and character $\chi$ on $\Gamma$ is a smooth function on the upper half plane with possible singularities at cusps that transforms like a modular form of weight $k$ and character $\chi$ on $\Gamma$ and is annihilated by the weight $k$ hyperbolic Laplacian
\[\Delta_k := -y^2\(\frac{\partial^2}{\partial x^2} + \frac{\partial^2}{\partial y^2}\) + iky\(\frac{\partial}{\partial x} + i\frac{\partial}{\partial y}\).\]
We denote by $H_{k,\chi}(\Gamma)$ the space of harmonic weak Maass forms of weight $k$ and character $\chi$ on $\Gamma$  (for the precise definition of harmonic weak Maass forms see Section \ref{section2.2}). 

For the differential operator $\xi_{2-k}(f)(z) := 2iy^{2-k}(\overline{\frac{\partial f}{\partial\bar{z}}})(z)$,  the assignment $f(z)\mapsto \xi_{2-k}(f)(z)$ gives an anti-linear mapping
\[\xi_{2-k}: H_{2-k,\chi}(\Gamma) \to M^!_{k,\bar{\chi}}(\Gamma),\]
where $M^!_{k,\bar{\chi}}(\Gamma)$ is the space of weakly holomorphic modular forms of weight $k$ and character $\bar{\chi}$ on $\Gamma$  (for the precise definition of weakly holomorphic modular forms see Section \ref{section2.1}). Let $H^*_{2-k,\chi}(\Gamma)$ be the inverse image of the space of cusp forms $S_{k,\bar{\chi}}(\Gamma)$ under the mapping $\xi_{2-k}$. Any harmonic weak Maass form $f(z)\in H^*_{2-k,\chi}(\Gamma)$ has a unique decomposition $f(z) = f^+(z) + f^-(z)$, where the function $f^{+}(z)$ (resp. $f^{-}(z)$)  is called the {\it{holomorphic}} (resp. {\it{non-holomorphic}}) part of $f(z)$. We denote   the space of holomorphic parts of $f(z)\in H^*_{2-k,\chi}(\Gamma)$ by $H^+_{2-k,\chi}(\Gamma)$.

On the other hand, a function $F(z)$ on $\HH$ is called an $\it{Eichler\ integral}$ of weight $2-k$ and character $\chi$ on $\Gamma$ if it satisfies
\[ ((D^{k-1}F)|_{k,\chi}\gamma)(z) = (D^{k-1}F)(z)\]
for all $\gamma=\sm a&b\\c&d\esm\in\Gamma$, where $(F|_{k,\chi}\gamma)(z) := \bar{\chi}(\gamma)(cz+d)^{-k}F(\gamma z)$ and  $D = \frac{1}{2\pi i}\frac{\partial}{\partial z}$. We let $E_{2-k,\chi}(\Gamma)$ denote the space of holomorphic Eichler integrals $F(z)$ of weight $2-k$ and character $\chi$ on $\Gamma$ such that
\begin{enumerate}
\item $F(z)$ is invariant under $|_{2-k,\chi}\gamma$ for all  translation matrices of the form $\sm 1& l\\ 0&1\esm \in\Gamma$,
\item $(D^{k-1}F)(z)$ can be written as
\begin{equation} \label{dfnofeichler}
(D^{k-1}F)(z) = g^*(z) + (D^{k-1}G)(z),
\end{equation}
where $g^*(z)$ is a supplementary function to a cusp form  $g(z) \in S_{k,\bar{\chi}}(\Gamma)$ and $G(z)\in M^!_{2-k,\chi}(\Gamma)$ (for the definition of a supplementary function, see Section \ref{section3.2}).
\end{enumerate}

Our first result shows that we can understand holomorphic parts of harmonic weak Maass forms in terms of Eichler integrals. Throughout the  paper, unless stated otherwise, 
we always assume that $k>2$ and $-I\in\Gamma$.

\begin{thm} \label{main1}  Let $H^+_{2-k,\chi}(\Gamma)$ and $E_{2-k,\chi}(\Gamma)$ be as above. Then 
\begin{equation*}
E_{2-k,\chi}(\Gamma)=
\begin{cases}
H^+_{2-k,\chi}(\Gamma) & \text{if $\kappa\neq0$},\\
H^+_{2-k,\chi}(\Gamma) + \CC & \text{if $\kappa=0$},
\end{cases}
\end{equation*}
where $\kappa\in [0,1)$ is an explicit constant  depending on $\chi$ and $\Gamma$ (see Section \ref{section2.1} for the precise definition).
\end{thm}

Now we   turn to period functions. A form $f(z)\in M^!_{k,\chi}(\Gamma)$ is a {\it{weakly holomorphic cusp form}}  if its constant term vanishes at every cusp of $\Gamma$. Let $S_{k,\chi}^!(\Gamma)$ denote the space of weakly holomorphic cusp forms.
For $f(z) = \sum_{n\gg-\infty}a_ne^{2\pi i(n+\kappa)z/\lambda}\in S^!_{k,\chi}(\Gamma)$, its {\it{Eichler integral}} is
\[\mc{E}_f(z) := \sum_{n\gg-\infty\atop  n+\kappa\neq0} a_n \biggl(\frac{n+\kappa}{\lambda}\biggr)^{-(k-1)}e^{2\pi i(n+\kappa)z/\lambda}.\]
We define the {\it{period function}} for $f(z)$ and $\gamma\in\Gamma$ by
\[r(f,\gamma;z) := c_{k}   (\mc{E}_f - \mc{E}_f|_{2-k,\chi}\gamma)(z),\]
where $c_k := -\frac{(k-2)!}{(2\pi i)^{k-1}}$.

On the other hand,  following \cite{BGK}, for each $\gamma = \sm a&b\\c&d\esm\in \Gamma$ and $\mc{F}^+(z)\in H^+_{2-k,\chi}(\Gamma)$, we define the {\it{$\gamma$-mock modular period function}} for  $\mc{F}^+(z)$ by
\[\mb{P}(\mc{F}^+,\gamma;z) := \frac{(4\pi)^{k-1}}{(k-2)!}(\mc{F}^+ - \mc{F}^+|_{2-k,\chi}\gamma)(z).\]
The following theorem is a generalization of \cite[Theorem 1.1]{BGK} showing that $\mb{P}(\mc{F}^+,\gamma;z)$ is related with the period function $r(\xi_{2-k}\mc{F},\gamma;z)$, and its coefficients encode critical values of $L$-functions associated to $f(z):= \xi_{2-k}\mc{F}(z)$.

\begin{thm} \label{main2} Let $\mc{F}(z)\in H^*_{2-k,\chi}(\Gamma)$ and $f(z) = \xi_{2-k}(\mc{F})(z) \in S_{k,\bar{\chi}}(\Gamma)$. Then
\[ [\mb{P}(\mc{F}^+,\gamma; \bar{z})]^- = \frac{1}{c_{k}} r(f,\gamma;z),\]
where $[\ ]^-$ indicates the complex conjugate of the function inside $[\ ]^-$.
   Moreover, if $c\neq 0$, then 
\[ \biggl[\mb{P}\biggl(\mc{F}^+,\gamma_{c,d}; \bar{z}-\frac dc\biggr)\biggr]^- = \sum_{j=0}^{k-2}\frac{L(f,\zeta_{c\lambda}^{-d},j+1)}{(k-2-j)!}(2\pi iz)^{k-2-j},\]
where $\gamma_{c,d}\in \Gamma$ is any matrix satisfying $\gamma_{c,d} = \sm *&*\\c&d\esm$. 
\end{thm}

Here, $L(f, \zeta^{-d}_{c\lambda}, s)$ is the twisted $L$-function of a cusp form $f(z)$. More precisely, let $f(z)$ be a cusp form in $S_{k,\chi}(\Gamma)$ with the Fourier expansion $f(z) = \sum_{n+\kappa>0} a_n e^{2\pi i(n+\kappa)z/\lambda}$.
Then for integers $c$ and $d$ with $(c,d)=1$ and $s\in\CC$
we consider the  series
\[L(f,\zeta^{-d}_{c\lambda}, s) = \sum_{n\in\ZZ\atop n+\kappa>0} \frac{a_n \zeta^{-d(n+\kappa)}_{c\lambda}}{((n+\kappa)/\lambda)^s}.\]
This series converges if $\Re(s)$ is sufficiently large. 
For such $s$  this series is the same as
\[\frac{(2\pi)^s}{\Gamma(s)}\int^{\infty}_0 f\left(iy-\frac dc\right)y^s \frac{dy}{y}.\]
But this integral gives an entire function, and hence we can have the analytic continuation of the series $L(f, \zeta^{-d}_{c\lambda}, s)$ on $\CC$. Throughout this paper, we consider $L(f, \zeta^{-d}_{c\lambda}, s)$ as its analytic continuation.

Furthermore, for $\mc{F}(z)\in H^*_{2-k,\chi}(\Gamma)$ and $\gamma\in\Gamma$, there are two natural periods $r(\xi_{2-k}(\mc{F}),\gamma;z)$ and $r(D^{k-1}(\mc{F}),\gamma;z)$, and they satisfy the following relation.

\begin{thm} \label{main3} If $\mc{F}(z)\in H^*_{2-k,\chi}(\Gamma)$ and $\gamma=\sm a&b\\c&d\esm\in\Gamma$, then 
\begin{align*}
r(\xi_{2-k}(\mc{F}),\gamma;z) =& \frac{(-4\pi)^{k-1}}{(k-2)!}\biggl\{[r(D^{k-1}(\mc{F}),\gamma;\bar{z})]^- +
\delta_{\kappa,0}\overline{c_{D^{k-1}(\mc{F})}c_k}\biggl(1-c^{k-2}\chi(\gamma)\biggl(z+\frac dc\biggr)^{k-2}\biggr)\biggr\},
\end{align*}
where $c_{D^{k-1}(\mc{F})}$ is an explicit constant depending on $\mc{F}(z)$ (see (\ref{constantformula}) for the precise definition).
 Moreover, if $\lambda=1$ and  $\Gamma$ is a subgroup of finite index of  the full modular group, then there is a function $\hat{\mc{F}}(z)\in H^*_{2-k,\chi}(\Gamma)$ for which $\xi_{2-k}(\hat{\mc{F}}) = \xi_{2-k}(\mc{F})$ and
\[r(\xi_{2-k}(\hat{\mc{F}}),\gamma;z) = (-1)^{k-1}\frac{(4\pi)^{k-1}}{(k-2)!}[r(D^{k-1}(\hat{\mc{F}}),\gamma;\bar{z})]^-.\]
\end{thm}
 Theorem \ref{main3} is a generalization of \cite[Theorem 1.4]{BGK}. However, by employing the 
 theory of supplementary functions  developed by M. I. Knopp and S. Y. Husseini \cite{Hus, Kon}, we can obtain an exact equation rather than a congruence.

\begin{ex} Let $\Gamma = \SL_2(\ZZ)$, $S = \sm 0&-1\\1&0\esm$, and $k=3$, $4$, or $5$. We define a multiplier system $\chi_{2k}$  by $\chi_{2k}(\gamma) := \frac{\eta^{2k}(\gamma z)}{\eta^{2k}(z)(cz+d)^{k}}$ for $\gamma = \sm a&b\\c&d\esm \in \SL_2(\ZZ)$. Then  $\eta^{2k} (z)$ is the unique cusp form (up to constant multiples) in the space $S_{k,\chi_{2k}}(\Gamma)$  (for example, see \cite{Leh}). 
If $\mc{F}(z)\in H^*_{2-k,\overline{\chi_{2k}}}(\Gamma)$, then $\xi_{2-k}(\mc{F})(z) \in S_{k,\chi_{2k}}(\Gamma)$ and its period polynomial $r(D^{k-1}(\mc{F}), S;z)$ lies in the space
\[W := \{ P(z) \in V_{k-2} |\ (P + P|_{2-k,\overline{\chi_{2k}}}S)(z) = (P+P|_{2-k,\overline{\chi_{2k}}}U + P|_{2-k,\overline{\chi_{2k}}}U^2)(z) = 0\},\]
where $V_{k-2}$ is the space of polynomials of degree at most $k-2$ in $z$ and $U := \sm 1&-1\\1&0\esm$.
After some computations, we find that the period polynomial $r(D^{k-1}(\mc{F}),S;z)$ is equal to
\begin{equation*}
\begin{cases}
c_{\mc{F}}(z+ i) & \text{if $k=3$},\\
c_{\mc{F}}(z^2 - \sqrt{3}iz - 1) & \text{if $k=4$},\\
c_{\mc{F}}(z^3 + \frac{3+\sqrt{3}}{2}iz^2 - \frac{3+\sqrt{3}}{2}z - i) & \text{if $k=5$},
\end{cases}
\end{equation*}
where $c_{\mc{F}}$ is a constant depending on $\mc{F}(z)$.  More precisely, the constant $c_{\mc{F}}$ is given by 
\[c_{\mc{F}} = i\left(R. \int^{i\infty}_0 (D^{k-1}\mc{F})(z)dz\right),\]
where $R.\int$ is the regularized integral introduced in \cite[Section 2]{BFK}.
\end{ex}

 The remainder of the paper is organized as follows. In Section \ref{section2}, we study harmonic weak Maass forms and derive some properties of those under $\xi_{2-k}$ and $D^{k-1}$. In Section \ref{section3}, we recall definitions and basic facts about Eichler integrals and we construct Poincar\'e series which give supplementary functions and Eichler integrals.  In Section \ref{section4}, we conclude with proofs of Theorems \ref{main1}, \ref{main2}, and \ref{main3}.

\section{Harmonic weak Maass forms} \label{section2}
Here we briefly recall definitions and basic facts about modular forms and harmonic weak Maass forms. For details, consult  \cite{Leh, Miy} for modular forms and \cite{BF, BOR} for harmonic weak Maass forms, for example.

\subsection{Modular forms} \label{section2.1}
Let $\Gamma$ be an {\it{$H$-group}}, i.e., a finitely generated Fuchsian group of the first kind which has at least one parabolic class. Let $k\in\ZZ$ and $\chi$ a (unitary) character of $\Gamma$. We  recall the useful slash operator
\[(F|_{k,\chi}\gamma)(z) = \bar{\chi}(\gamma)(cz+d)^{-k}F(\gamma z)\]
for any function $F(z)$ and $\gamma = \sm a&b\\c&d\esm\in\Gamma$.
Let $T = \sm 1&\lambda\\ 0&1\esm,\ \lambda>0$, generate the subgroup $\Gamma_\infty$ of translations in $\Gamma$. If $F(z)$ satisfies $(F|_{k,\chi}T)(z) = F(z)$, then
\[F(z+\lambda) = \chi(T)F(z) = e^{2\pi i\kappa}F(z)\]
with $0\leq \kappa <1$. Thus, if $F(z)$ is holomorphic in $\HH$, then $F(z)$ has the Fourier expansion at $i\infty$ (actually a Laurent expansion)
\begin{equation} \label{Fourier1}
F(z) = \sum_{n= -\infty}^\infty a_ne^{[2\pi i(n+\kappa)z]/\lambda}.
\end{equation}
Suppose that in addition to $i\infty$, $\Gamma$ has $t\geq0$ inequivalent parabolic classes. Each of these classes corresponds to a cyclic subgroup of parabolic elements in $\Gamma$ leaving fixed a parabolic cusp on the boundary of $\HH$. Such a parabolic cusp lies on the real axis. Let $q_1,\ldots, q_t$ be the inequivalent parabolic cusps (other than $i\infty$) on the boundary of $\HH$ and let $\Gamma_j$ be the cyclic subgroup of $\Gamma$ fixing $q_j,\ 1\leq j\leq t$. Suppose also that
\begin{equation*}
Q_j = \sm *&*\\ c_j&d_j\esm,\ 1\leq j\leq t,
\end{equation*}
is a generator of $\Gamma_j$. For $1\leq j\leq t$, put $\chi(Q_j) = e^{2\pi i\kappa_j},\ 0\leq \kappa_j<1$. If a holomorphic function $F(z)$ satisfies $(F|_{k,\chi}Q_j)(z) = F(z)$, then $F(z)$ has the Fourier expansion at $q_j$:
\begin{equation}  \label{Fourier2}
F(z) = \biggl(\frac{-1}{\lambda_j(z-q_j)}\biggr)^{k}\sum_{n=-\infty}^\infty a_n(j)e^{[-2\pi i(n+\kappa_j)]/[\lambda_j(z-q_j)]},
\end{equation}
where $\lambda_j$ is a positive real number called the $\it{width\ of\ the\ cusp}$ $q_j$ and defined as  follows. Let $A_j = \sm 0&-1\\1&-q_j\esm$, so that $A_j$ has determinant $1$ and $A_j(q_j) = \infty$. Then $\lambda_j>0$ is chosen so that
\[A_j^{-1}\sm 1&\lambda_j\\0&1\esm A_j\]
generates $\Gamma_j$, the stabilizer of $q_j$.

We are now in a position to give the following:
\begin{dfn}
Suppose $F(z)$ is holomorphic in $\HH$ and satisfies the functional equation
\[(F|_{k,\chi}\gamma)(z) = F(z)\]
for all $\gamma\in\Gamma$.
\begin{enumerate}
\item[(1)] If $F(z)$ has only finitely many terms with $n+\kappa<0$ in (\ref{Fourier1})  and with $n+\kappa_j<0,\ 1\leq j\leq t$, in (\ref{Fourier2}), then $F(z)$ is called a weakly holomorphic modular form of weight $k$ and character $\chi$ on $\Gamma$. The set of all such weakly holomorphic modular forms is denoted by $M^!_{k,\chi}(\Gamma)$.

\item[(2)] Let $F(z)\in M^!_{k,\chi}(\Gamma)$. Suppose in addition $F(z)$ has only terms with $n+\kappa\geq0$ in (\ref{Fourier1}) and $n+\kappa_j\geq0,\ 1\leq j\leq t$, in (\ref{Fourier2}). Then $F(z)$ is called a holomorphic modular form. The set of holomorphic modular forms in $M^!_{k,\chi}(\Gamma)$ is denoted by $M_{k,\chi}(\Gamma)$.

\item[(3)] If $F(z)\in M_{k,\chi}(\Gamma)$ and has only terms with $n+\kappa>0,\ n+\kappa_j>0$ in the expansions (\ref{Fourier1}), (\ref{Fourier2}), respectively, then $F(z)$ is called a cusp form. The collection of cusp forms in $M_{k,\chi}(\Gamma)$ is denoted by $S_{k,\chi}(\Gamma)$.
\end{enumerate}
\end{dfn}

\subsection{Harmonic weak Maass forms} \label{section2.2}
 We start with the definition of harmonic weak Maass forms.

\begin{dfn} A harmonic weak Maass form of weight $k$ and character $\chi$ on $\Gamma$ is any smooth function on $\HH$ satisfying
\begin{enumerate}
\item[(1)] $(f|_{k,\chi}\gamma)(z) = f(z)$ for all $\gamma\in\Gamma$,
\item[(2)] $\Delta_k f =0$,
\item[(3)] a linear exponential growth condition in terms of $y$  at every cusp.
\end{enumerate}
We write $H_{k,\chi}(\Gamma)$ for the space of harmonic weak Maass forms of weight $k$ and character $\chi$ on $\Gamma$.
\end{dfn}

Recall that $T = \sm 1&\lambda\\ 0&1\esm,\ \lambda>0$, is a generator of $\Gamma_{\infty}$ and $\chi(T) = e^{2\pi i\kappa}$.
The transformation property (1) implies that $f(z)\in H_{k,\chi}(\Gamma)$ has the Fourier expansion
\[f(z) = \sum_{n\gg-\infty} a_n(y)e^{2\pi i(n+\kappa)x/\lambda}.\]
Since $\Delta_k f = 0$, the coefficients $a_n(y)$ satisfy the second order differential equation
\[\Delta_k a_n(y)e^{2\pi i(n+\kappa)x/\lambda}=0\]
as functions in $y$.  To describe $a_n(y)$, we consider the function
\[ H(k;w) = e^{-w}\int^\infty_{-2w}e^{-t}t^{-k}dt.\]
The integral converges for $k<1$ and can be holomorphically continued in $k$ (for $w\neq0$) in the same way as the Gamma function. If $w<0$, then $H(k;w) = e^{-w}\Gamma(1-k, -2w)$, where $\Gamma(a,x)$ denotes the incomplete Gamma function as in \cite{Abr}. We find that
\begin{equation*}
a_n(y) =
\begin{cases}
a^+_0 + a^-_0y^{1-k}, & \text{if $n+\kappa =0$},\\
a^+_ne^{-2\pi (n+\kappa)y/\lambda} + a^-_n H(k;2\pi (n+\kappa)y/\lambda), & \text{if $n+\kappa\neq0$},
\end{cases}
\end{equation*}
with complex coefficients $a^{\pm}_n$. Thus any harmonic weak Maass form $f(z)$ of weight $k$ has the unique decomposition $f(z) = f^+(z) + f^-(z)$, where
\begin{align} \label{decomposition}
f^+(z) &= \sum_{n\gg-\infty} a^+_n e^{2\pi i(n+\kappa)z/\lambda},\\
\nonumber f^-(z) &= \delta_{\kappa,0}a^-_0y^{1-k} + \sum_{n\ll \infty\atop n+\kappa\neq 0} a^-_n  H(k;2\pi (n+\kappa)y/\lambda)e^{2\pi i(n+\kappa)x/\lambda},
\end{align}
where $\delta_{\kappa,0} = 1$ if $\kappa = 0$, and $\delta_{\kappa,0} =0$ otherwise.

\subsection{Differential operators} \label{section2.3}
We introduce the Maass raising and lowering operators on non-holomorphic modular forms of weight $k$ as
\[R_k = 2i\frac{\partial}{\partial z} + ky^{-1}\;\; \text{and}\;\; L_k = -2iy^2\frac{\partial}{\partial \bar{z}}.\]
The following theorem is a summary of results on how  differential operators $\xi_{2-k}$ and $D^{k-1}$ act on the space of harmonic weak Maass forms.

\begin{thm} \cite[Proposition 3.2]{BF}, \cite[Theorem 1.2]{BOR} \label{operator}
 Let $k\in\ZZ, k>2$ and let $f(z) \in H_{2-k,\chi}(\Gamma)$.
\begin{enumerate}
\item[(1)]  The assignment $f(z) \mapsto \xi_{2-k}(f)(z) := y^{-k}\overline{L_{2-k} f(z)} = R_{k-2}y^{2-k}\overline{f(z)}$ defines an anti-linear mapping
\[\xi_{2-k} : H_{2-k,\chi}(\Gamma) \to M^!_{k,\bar{\chi}}(\Gamma).\]
Moreover, the kernel of $\xi_{2-k}$ is $M^!_{2-k,\chi}(\Gamma)$.

\item[(2)]  If we let $D:= \frac{1}{2\pi i}\frac{\partial}{\partial z}$, then $D^{k-1}$ defines a linear map
\[D^{k-1}: H_{2-k,\chi}(\Gamma) \to M^!_{k,\chi}(\Gamma).\]
\end{enumerate}
\end{thm}

We let $H^*_{2-k,\chi}(\Gamma)$ denote the inverse image of the space of cusp forms $S_{k,\bar{\chi}}(\Gamma)$ under the mapping $\xi_{2-k}$. Hence, if $f(z)\in H^*_{2-k,\chi}(\Gamma)$, then the Fourier coefficients $a^-_n$ vanish if $n+\kappa\geq0$.

\section{Eichler integrals} \label{section3}

In this section, we introduce the basic notions of  Eichler integrals and show how to construct the Eichler integrals for a given cusp form by using the supplementary function.

\subsection{Eichler integrals} \label{section3.1}

A result of Bol \cite{Bol} states that
\begin{equation} \label{Bol}
D^{k-1}[(cz+d)^{k-2}F(\gamma z)] = (cz+d)^{-k}(D^{k-1}F)(\gamma z)
\end{equation}
for any $\gamma = \sm a&b\\c&d\esm \in {\rm SL}_2 (\mathbb{R})$ and any function $F(z)$ with sufficiently many derivatives. From \eqref{Bol} we see that if $F(z)\in M^!_{2-k,\chi}(\Gamma)$, then $(D^{k-1}F)(z)\in M^!_{k,\chi}(\Gamma)$. Furthermore, if $f(z)\in M^!_{k,\chi}(\gamma)$ and $F(z)$ is any $(k-1)$-fold indefinite integral of $f(z)$, then $F(z)$ satisfies
\begin{equation} \label{period}
(F|_{2-k,\chi}\gamma)(z) = F(z) + p_\gamma(z),\ \gamma\in\Gamma,
\end{equation}
where $p_\gamma(z)$ is a polynomial in $z$ of degree at most $k-2$.

\begin{dfn} If $F(z)$ is  an Eichler integral of weight $2-k$ and character $\chi$ on $\Gamma$, then the functions $p_\gamma(z)$ occurring in (\ref{period})  are called the period functions of $F(z)$ $($or of $(D^{k-1}F)(z))$.
\end{dfn}

For example, we suppose $f(z)\in S_{k,\chi}(\Gamma)$ and define
\begin{equation} \label{Eichler1}
\mc{E}_f(z) := \frac{1}{c_k}\int_z^{i\infty} f(\tau)(z-\tau)^{k-2}d\tau
\end{equation}
and
\begin{equation} \label{Eichler2}
\mc{E}^N_f(z) := \frac{1}{c_k} \biggl[\int_z^{i\infty} f(\tau)(\bar{z}-\tau)^{k-2}d\tau\biggr]^-,
\end{equation}
where  $c_k := -\frac{(k-2)!}{(2\pi i)^{k-1}}$ and
 $[\ ]^-$ indicates the complex conjugate of the function inside $[\ ]$.
    Their period functions  can be written explicitly and they satisfy a certain relation.  Integrals of the type defined by (\ref{Eichler2}) were studied by W. Pribitkin, who called them ``auxiliary integrals'' and also noted their modular behavior (for example, see \cite{Pri}).

\begin{prop} \cite[Lemma 2.2]{KM} \label{computationperiod}
Let $\mc{E}_f(z)$ and $\mc{E}^N_f(z)$ be the Eichler integrals defined by \eqref{Eichler1} and \eqref{Eichler2}. Then
\begin{equation} \label{p1}
r(f,\gamma;z) := c_k(\mc{E}_f - \mc{E}_f|_{2-k,\chi}\gamma)(z) =  \int^{i\infty}_{\gamma^{-1}(i\infty)}f(\tau)(z-\tau)^{k-2}d\tau
\end{equation}
and
\begin{equation} \label{p2}
 r^N(f,\gamma;z) := c_k(\mc{E}^N_f-\mc{E}^N_f|_{2-k,\bar{\chi}}\gamma)(z) = \biggl[\int^{i\infty}_{\gamma^{-1}(i\infty)}f(\tau)(\bar{z}-\tau)^{k-2}d\tau\biggr]^-
\end{equation}
for all $\gamma\in\Gamma$. In particular, $r(f,\gamma;z) = [r^N(f,\gamma;\bar{z})]^-$ for all $\gamma\in\Gamma$.
\end{prop}

\subsection{Supplementary function} \label{section3.2}
 Suppose that $k\in\ZZ, k>2$. Let $m$ be an integer and consider the Poincar\'e series
\begin{equation}\label{poin}
g_m(z,\chi) := \sum_{\gamma} \frac{e^{2\pi i(m+\kappa)\gamma z/\lambda}}{\chi(\gamma)(cz+d)^{k}},
\end{equation}
where $\gamma = \sm *&*\\c&d\esm$ runs through a complete set of elements of $\Gamma$ with distinct lower rows.
The following properties of the Poincar\'e series are well known:

\begin{thm} \label{poincare} \cite[pp. 272--289]{Leh} For the Poincar\'e series $g_m(z,\chi)$ defined by \eqref{poin}, we have the following.
\begin{enumerate}
\item[(1)] $g_m(z,\chi)\in M^!_{k,\chi}(\Gamma)$.
\item[(2)] $g_m(z,\chi)$ vanishes at all cusps of $\Gamma$ except possibly at $i\infty$. At $i\infty$ it has an expansion of the form
\[g_m(z,\chi) = 2e^{2\pi i(m+\kappa)z/\lambda}+2\sum_{n+\kappa>0}a_n(m,\chi)e^{2\pi i(n+\kappa)z/\lambda}.\]
Thus, if $m+\kappa>0$, then $g_m(z,\chi)\in S_{k,\chi}(\Gamma)$.
\item[(3)] There exist integers $0\leq m_1 < \ldots < m_s$ such that $g_{m_1}(z,\chi), \ldots, g_{m_s}(z,\chi)$ form a basis for $S_{k,\chi}(\Gamma)$.
\end{enumerate}
\end{thm}

 Now we recall the theory of supplementary functions (for details see \cite{Kon, KL}).
For $f(z)\in S_{k,\chi}(\Gamma)$,  by (3) of Theorem \ref{poincare}, there exist complex numbers $b_1,\ldots, b_s$ such that $f(z) = \sum_{i=1}^s b_ig_{m_i}(z,\chi)$. Put $f^*(z) = \sum_{i=1}^s\overline{b_i}g_{m_i'}(z,\bar{\chi})$, where
\begin{equation*}
m' =
\begin{cases}
-m, & \text{if $\kappa =0$},\\
-1 - m, & \text{if $\kappa>0$}.
\end{cases}
\end{equation*}
 Recalling that $\chi(T) = e^{2\pi i\kappa},\ 0\leq\kappa<1$,  we see that $\bar{\chi}(T) = e^{2\pi i\kappa'},\ 0\leq \kappa' <1$, where
\begin{equation} \label{kappa}
\kappa' =
\begin{cases}
0, & \text{if $\kappa = 0$},\\
1-\kappa, & \text{if $\kappa>0$}.
\end{cases}
\end{equation}
Thus we have the expansion at $i\infty$
\begin{align*}
g_{m_i'}(z,\bar{\chi}) &= 2e^{2\pi i(m_i'+\kappa')z/\lambda}+2\sum_{n+\kappa'>0}a_n(m_i',\bar{\chi})e^{2\pi i(n+\kappa')z/\lambda}\\
&= 2e^{-2\pi i(m_i+\kappa)z/\lambda} + 2\sum_{n+\kappa'>0}a_n(m_i',\bar{\chi})e^{2\pi i(n+\kappa')z/\lambda}.
\end{align*}
It follows that $f^*(z)\in M^!_{k,\bar{\chi}}(\Gamma)$, $f^*(z)$ has a pole at $i\infty$ with the principal part
\[2\sum_{i=1}^s \overline{b_i}e^{-2\pi i(m_i+\kappa)z/\lambda}\]
and $f^*(z)$ vanishes at all of the other cusps of $\Gamma$. We call $f^*(z)$ the {\it{supplementary\ function to}} $f(z)$.  Note that the supplementary function $f^*(z)$ is not unique since it depends on the representation of $f(z)$ as a sum of Poincar\'e series and there are relations between the Poincar\'e series.

A form $f(z)\in M^!_{k,\chi}(\Gamma)$ is a {\it{weakly holomorphic cusp form}} if its constant term vanishes at every cusp of $\Gamma$. Let $S_{k,\chi}^!(\Gamma)$ denote the space of weakly holomorphic cusp forms.
 Forms of the type $S_{k,\chi}^!(\Gamma)$ were studied earlier by W. Pribitkin, who called them ``constant-free modular forms'' and also examined certain integrals associated to them (see \cite{Pri}).
If $f(z)=\sum_{n+\kappa>0}a_ne^{2\pi i(n+\kappa)\tau/\lambda}\in S_{k,\chi}(\Gamma)$, then its Eichler integral is
\[\mc{E}_f(z) = \sum_{n+\kappa>0} a_n\biggl(\frac{n+\kappa}{\lambda}\biggr)^{-(k-1)}e^{2\pi i(n+\kappa)\tau/\lambda}.\]
Now we extend this definition to $f(z) = \sum_{n+\kappa\gg-\infty}a_ne^{2\pi i(n+\kappa)\tau/\lambda}\in S^!_{k,\chi}(\Gamma)$ as
\[\mc{E}_f(z) = \sum_{n+\kappa\gg-\infty} a_n\biggl(\frac{n+\kappa}{\lambda}\biggr)^{-(k-1)}e^{2\pi i(n+\kappa)\tau/\lambda}.\]

In \cite{Leh0}, using the circle method, J. Lehner showed that the Fourier coefficients of modular forms of negative weight are completely determined by the principal parts of the expansions of those forms at the cusps. Hence, using the information about the principal part of $\mc{E}_f(z)$, we can define the constant term $c_f$ associated with $\mc{E}_f(z)$. For example, if we assume that $f(z)$ has a pole at $i\infty$ and that it is holomorphic at all other cusps, then $c_f$ is equal to
\begin{equation} \label{constantformula}
c_f := \frac{1}{\lambda(k-1)!}\sum_{l<0}\sum_{\gamma=\sm a&b\\c&d\esm\in C^+}a_l \biggl(\frac{-2\pi i}{c}\biggr)^k\chi^{-1}(\gamma)e^{\frac{2\pi i}{c\lambda}la},
\end{equation}
where
\[C^+ := \{\sm a&b\\c&d\esm\in\Gamma|\ c>0,\ 0\leq -d,\ a < c\lambda\}.\]
 Note that $c_f$ is a sum of Kloosterman sums, and this Kloosterman sum is essentially the constant coefficient of a Poincar\'e series for each $l$.
With this constant term $c_f$, we define another Eichler integral and period functions of $f(z)$ as follows:
\[\mc{E}^H_f(z) := \mc{E}_f(z) + \delta_{\kappa,0}c_f\]
and
\[r^H(f,\gamma;z)  := c_k(\mc{E}^H_f - \mc{E}^H_f|_{2-k,\chi}\gamma)(z).\]
 In addition, we define $c_f = 0$ if $f(z)$ is a cusp form. Therefore, we see that $r(f,\gamma;z) = r^H(f,\gamma;z)$ for a cusp form $f(z)$.

The following proposition describes the properties of $\mc{E}^{H}$ and $r^{H}$.

\begin{thm} \cite[Section 2]{Hus}, \cite[Theorem 1]{Leh0} \label{suppleperiod} Suppose that $k>2$ is an integer.
Then we have the following.
\begin{enumerate}
\item[(1)] If $f(z)\in S_{k,\bar{\chi}}(\Gamma)$, then $r(f,\gamma;z) = [r^H(f^*,\gamma;\bar{z})]^-$ for all $\gamma\in\Gamma$.
\item[(2)] If $f(z)\in M^!_{2-k,\chi}(\Gamma)$, then $\mc{E}^H_{D^{k-1}f}(z) = f(z)$.
\end{enumerate}
\end{thm}

 Actually, part (1) follows quite readily from a result proved by M. I. Knopp in \cite{Kon}, which appeared in 1962. One can also see the paper of M. I. Knopp and J. Lehner \cite{KL} for this result.

\bigskip

\section{Proofs of  the main theorems} \label{section4}

In this section, we prove  the main theorems. First, we prove Theorem \ref{main1} via the supplementary functions associated to cusp forms.

\begin{proof} [\bf Proof of Theorem \ref{main1}]
Suppose that $\kappa=0$. 
 Let $a\in\CC$. Then $a$ is periodic and $D^{k-1}a = 0$. Therefore, $a$ is the Eichler integral in $E_{2-k,\chi}(\Gamma)$. In this case, the corresponding $g^*(z)$ and $G(z)$ in (\ref{dfnofeichler}) are zero functions.
Therefore we can see that $\CC\subset E_
{2-k,\chi}(\Gamma)$.

 For $\mc{F}^+(z)\in H^+_{2-k,\chi}(\Gamma)$, by the definition of $H^+_{2-k,\chi}(\Gamma)$, there is a non-holomorphic function $\mc{F}^-(z)$ such that
\[\mc{F}(z):= \mc{F}^+(z) + \mc{F}^-(z)\in H^*_{2-k,\chi}(\Gamma).\]
Then $h(z) := (\xi_{2-k}\mc{F}^-)(z)\in S_{k,\bar{\chi}}(\Gamma)$. Let $h^*(z)$ be its supplementary function. We define
\[\mc{H}^+(z) := \mc{E}^H_{h^*}(z),\ \mc{H}^-(z) := -\mc{E}^N_{h}(z).\]
By Proposition \ref{computationperiod} and  Theorem \ref{suppleperiod}, we find that
\[ r^H(h^*,\gamma;z) = [r^H(h,\gamma;\bar{z})]^-  =  r^N(h,\gamma;z),\]
which implies that $\mc{H}(z) =\mc{H}^+(z) + \mc{H}^-(z)$ is invariant under the slash  operator $|_{2-k,\chi}\gamma$ for all $\gamma\in\Gamma$. We can also check that $\Delta_{2-k}(\mc{H}) = 0$. Therefore, $\mc{H}(z) \in H^*_{2-k,\chi}(\Gamma)$. Note that
\[(\xi_{2-k}\mc{H}^-)(z) = \frac{(-4\pi)^{k-1}}{(k-2)!}h(z)\]
and hence $\xi_{2-k}\(\mc{F}-\frac{(k-2)!}{(-4\pi)^{k-1}}\mc{H}\)(z) =0$. Therefore, we have arrived at
\[G(z) := \biggl(\mc{F}-\frac{(k-2)!}{(-4\pi)^{k-1}}\mc{H}\biggr)(z) \in M^!_{2-k,\chi}(\Gamma)\]
by Theorem \ref{operator}.  
 By the definition of the supplementary function, the constant term of $h^*(z)$ is zero, and hence we see that
$D^{k-1}(\mc{E}^H_{h^*})(z) = h^*(z)$. This implies that
\[D^{k-1}(\mc{F}^+) = D^{k-1}(\mc{F}) = \frac{(k-2)!}{(-4\pi)^{k-1}}h^*(z) + (D^{k-1}G)(z).\]
  Since the function $\frac{(k-2)!}{(-4\pi)^{k-1}}h^*(z)$ is a supplementary function to a cusp form $\frac{(k-2)!}{(-4\pi)^{k-1}}h(z)\in S_{k,\bar{\chi}}(\Gamma)$, we can conclude that $H^+_{2-k,\chi}(\Gamma) \subset E_{2-k,\chi}(\Gamma)$.

Conversely suppose that $F(z)\in E_{2-k,\chi}(\Gamma)$. Then by the definition of $E_{2-k,\chi}(\Gamma)$, we can decompose $(D^{k-1}F)(z)$ as follows:
\[(D^{k-1}F)(z) = g^*(z) + (D^{k-1}G)(z),\]
where $g^*(z)$ is a supplementary function to a cusp form $g(z)\in S_{k,\bar{\chi}}(\Gamma)$ and $G(z) \in M^!_{2-k,\chi}(\Gamma)$. If we define $\mc{H}^{+}(z)$ and $\mc{H}^{-}(z)$ as
\begin{equation} \label{constructionharmonic}
\mc{H}^+(z) := \mc{E}^H_{g^*}(z),\  \mc{H}^-(z) := -\mc{E}^N_g(z), 
\end{equation}
then, we already checked that $\mc{H}^+(z) + \mc{H}^-(z)\in H^*_{2-k,\chi}(\Gamma)$. Since $G(z)\in M^!_{2-k,\chi}(\Gamma) \subset H^*_{2-k,\chi}(\Gamma)$, we see that
\[\mc{H}(z):= \mc{H}^+(z)+\mc{H}^-(z) + G(z)\in H^*_{2-k,\chi}(\Gamma).\]
Moreover, since the holomorphic part of $\mc{H}(z)$ is $\mc{H}^+(z) + G(z)$ and
\[D^{k-1}(\mc{H}^+ + G)(z) = g^*(z) + (D^{k-1}G)(z) = D^{k-1}(F)(z),\]
we deduce that $F(z) = (\mc{H}^+ + G)(z) + c$ for some  constant $c\in\CC$.  
We used  the fact that if a polynomial is periodic then it must be a constant.

Now suppose that $\kappa\neq0$. Note that in this case there is no constant term in the Fourier expansion.
By an argument similar to that in the case of $\kappa =0$, we can check that
\[H^+_{2-k,\chi}(\Gamma) = E_{2-k,\chi}(\Gamma),\]
which completes the proof.
\end{proof}

Theorem \ref{main1} plays a key role in the proof of Theorems \ref{main2} and \ref{main3}.

\begin{proof} [\bf Proof of Theorem \ref{main2}]

Let $\mc{F}(z) = \mc{F}^+(z) + \mc{F}^-(z)\in H^*_{2-k,\chi}(\Gamma)$. By Theorem \ref{main1}, $\mc{F}^+(z)\in E_{2-k,\chi}(\Gamma)$, and this implies that
\[(D^{k-1}\mc{F}^+)(z) = g^*(z) + (D^{k-1}G)(z),\]
where $g^*(z)$ is a supplementary function to a cusp form $g(z)\in S_{k,\bar{\chi}}(\Gamma)$ and $G(z)\in M^!_{2-k,\chi}(\Gamma)$. As we saw in (\ref{constructionharmonic}), we have 
\[\mc{F}^-(z) = -\mc{E}^N_g(z).\]
Therefore, we deduce that
\begin{align*}
\mb{P}(\mc{F}^+,\gamma;z) &= \frac{(4\pi)^{k-1}}{(k-2)!}(\mc{F}^+-\mc{F}^+|_{2-k,\chi}\gamma)(z)\\
&=-\frac{(4\pi)^{k-1}}{(k-2)!}(\mc{F}^--\mc{F}^-|_{2-k,\chi}\gamma)(z)\\
&= \frac{(4\pi)^{k-1}}{(k-2)!}(\mc{E}^N_g-\mc{E}^N_g|_{2-k,\chi}\gamma)(z)\\
&= \frac1{c_k}\frac{(4\pi)^{k-1}}{(k-2)!}r^N(g,\gamma;z) = \frac1{c_k}\frac{(4\pi)^{k-1}}{(k-2)!}[r(g,\gamma;\bar{z})]^-.
\end{align*}
 The last equality follows from Proposition \ref{computationperiod}.
 Note that $f(z)\in S_{k,\bar{\chi}}(\Gamma)$ is defined by $f(z) = \xi_{2-k}(\mc{F})$ in the statement of Theorem \ref{main2}.
Since $\xi_{2-k}(\mc{F}^-) = \frac{(-4\pi)^{k-1}}{(k-2)!}g(z) = f(z)$, we find that
\[r(g,\gamma;z) = \frac{(k-2)!}{(-4\pi)^{k-1}}r(f,\gamma;z),\]
and hence we have 
\begin{align*}
[\mb{P}(\mc{F}^+,\gamma;\bar{z})]^- &= \frac{1}{\overline{c_k}}\frac{(4\pi)^{k-1}}{(k-2)!}r(g,\gamma;z)= \frac{(-1)^{k-1}}{\overline{c_k}}r(f,\gamma;z)= \frac{1}{c_k}r(f,\gamma;z),
\end{align*}
 where $c_k := -\frac{(k-2)!}{(2\pi i)^{k-1}}$.

 In the case of $\gamma = \gamma_{c,d}$, by Proposition \ref{computationperiod} we have 
\begin{align*}
[\mb{P}(\mc{F}^+,\gamma_{c,d};\bar{z})]^- =& \frac{1}{c_k}r(f,\gamma_{c,d};z)= \frac{1}{c_k}\int^{i\infty}_{-\frac dc} f(\tau)(z-\tau)^{k-2}d\tau.
\end{align*}
From this we see that
\begin{align*}
[\mb{P}(\mc{F}^+,\gamma_{c,d};\bar{z})]^- =& \frac{1}{c_k}\sum_{j=0}^{k-2}i^{-j+1}\mat k-2\\j\emat\biggl(\int^\infty_0 f\biggl(it-\frac dc\biggr)t^jdt\biggr)\biggl(z+\frac dc\biggr)^{k-2-j}\\
=& \sum_{j=0}^{k-2}\frac{L(f,\zeta_{c\lambda}^{-d},j+1)}{(k-2-j)!}\biggl(2\pi i\biggl(z+\frac dc\biggr)\biggr)^{k-2-j}.
\end{align*}
Therefore, by the change of variable $z\mapsto z-\frac dc$, we get the desired result
\[\biggl[\mb{P}\left(\mc{F}^+,\gamma_{c,d};\bar{z}-\frac dc\right)\biggr]^- = \sum_{j=0}^{k-2}\frac{L(f,\zeta_{c\lambda}^{-d},j+1)}{(k-2-j)!}(2\pi iz)^{k-2-j}.\]
This completes the proof.
\end{proof}

To prove Theorem \ref{main3} we need following lemmas.  For the first lemma, let $Id$ be the trivial character of $\SL_2(\ZZ)$. 
Then the constant coefficients of the forms in $M^!_{2,Id}(\SL_2(\ZZ))$ are always zero since $M^!_{2, Id}(\SL_2(\ZZ))$ is equal to the set of derivatives of  weakly holomorphic modular forms with weight $0$. 
Therefore, by the result of Borcherds in \cite{Bor} we have the following lemma.

\begin{lem}\cite[Theorem 3.1]{Bor} \label{existence} Let $F(z) = \sum_{n\gg -\infty}a_ne^{2\pi inz}$. Then $F(z)\in M^!_{2,Id}(\SL_2(\ZZ))$ if and only if
\[ \sum_{n\in\ZZ\atop n\neq 0} a_n b_{-n} = 0,\]
for all $\sum_{n\gg-\infty}b_ne^{2\pi inz}\in M^!_{0,Id}(\SL_2(\ZZ))$.
\end{lem}

\begin{lem} \label{existence2}  Let $k\in\ZZ, k>2$.
 If $\kappa=0,\ \lambda=1$ and  $\Gamma$ is a subgroup of finite index of the full modular group, then there are forms in $M^!_{2-k,\chi}(\Gamma)$ with non-zero constant terms.
\end{lem}

\begin{proof} [\bf Proof of Lemma \ref{existence2}]
By the Riemann-Roch theorem that there is a non-zero 
weakly holomorphic modular form $F(z)=\sum_{n\gg -\infty}a_ne^{2\pi inz}$  in $M^!_{2-k,\chi}(\Gamma)$. 
Suppose that there is no weakly holomorphic modular form in $M^!_{2-k,\chi}(\Gamma)$ 
with non-zero constant term. For every weakly holomorphic modular form $G(z) := \sum_{n\gg-\infty}b_ne^{2\pi inz}\in M^!_{0,Id}(\SL_2(\ZZ))$, we see that $F(z)G(z)\in M^!_{2-k,\chi}(\Gamma)$. By assumption, the constant term of $F(z)G(z)$
is zero, i.e.,
\begin{equation} \label{constantzero}
\sum_{n\in\ZZ} a_n b_{-n} = 0.
\end{equation}
On the other hand, we can consider $F(z)$ as a formal power series as in Lemma \ref{existence}. Then, since  (\ref{constantzero})  holds for every $\sum_{n\gg-\infty}b_ne^{2\pi inz}\in M^!_{0,Id}(\SL_2(\ZZ))$, by Lemma \ref{existence}, we see that $F(z)$ is a weakly holomorphic modular form in $M^!_{2,Id}(\SL_2(\ZZ))$.  However, $F(z)$ is already a weakly holomorphic modular form of weight $2-k$. Since $k>2$,  we see that $2$ and $2-k$ can not be the same. Therefore, $F(z)$ should be  identically zero and this is a contradiction by the assumption that $F(z)$ is non-zero.
In conclusion, there is a weakly holomorphic modular form $F(z)\in M^!_{2-k,\chi}(\Gamma)$ whose constant term is not zero.
\end{proof}

We are now ready to prove Theorem \ref{main3}.

\begin{proof} [\bf Proof of Theorem \ref{main3}]
Suppose that $\mc{F}(z) = \mc{F}^+(z) + \mc{F}^-(z)\in H^*_{2-k,\chi}(\Gamma)$. By Theorem \ref{main1},
\[(D^{k-1}\mc{F})(z) =
(D^{k-1}\mc{F}^+)(z) = g^*(z) + (D^{k-1}G)(z),\]
where $g^*(z)$ is a supplementary function to a cusp form $g(z)\in S_{k,\bar{\chi}}(\Gamma)$ and $G(z)\in M^!_{2-k,\chi}(\Gamma)$. We also observe that
\[\mc{F}^-(z) = -\mc{E}^N_g(z) \quad \text{and} \quad \xi_{2-k}(\mc{F})(z) = \xi_{2-k}(\mc{F}^-)(z) = \frac{(-4\pi)^{k-1}}{(k-2)!}g(z).\]
Therefore,  if we use Theorem \ref{suppleperiod}, then we find that
\[r(\xi_{2-k}\mc{F},\gamma;z) = \frac{(-4\pi)^{k-1}}{(k-2)!}r(g,\gamma;z) =  \frac{(-4\pi)^{k-1}}{(k-2)!}[r^H(g^*,\gamma;\bar{z})]^-\]
and
\[r^H(D^{k-1}(\mc{F}),\gamma;z) = r^H(g^*,\gamma;z) + r^H(D^{k-1}G,\gamma;z).\]
Since $G(z) \in M^!_{2-k,\chi}(\Gamma)$, we see that
$\mc{E}^H_{D^{k-1}(G)}(z) = G(z)\in M^!_{2-k,\chi}(\Gamma)$ by  Theorem \ref{suppleperiod}, and hence we have 
\[r^H(D^{k-1}G,\gamma;z) = c_k(\mc{E}^H_{D^{k-1}(G)} - \mc{E}^H_{D^{k-1}(G)}|_{2-k,\chi}\gamma)(z) =   0\]
for all $\gamma\in\Gamma$. Therefore, we arrive at
\[r^H(D^{k-1}(\mc{F}),\gamma;z) =  r^H(g^*,\gamma;z) =  \frac{(k-2)!}{(-4\pi)^{k-1}}[r(\xi_{2-k}\mc{F},\gamma;\bar{z})]^-.\]
 By an easy calculation based on the definition of $r^H$, we get
\begin{align*}
r^H(D^{k-1}(\mc{F}),\gamma;z)  =& r(D^{k-1}(\mc{F}),\gamma;z) +\delta_{\kappa,0} c_{D^{k-1}(\mc{F})}c_k\biggl(1-c^{k-2}\bar{\chi}(\gamma)\biggl(z+\frac dc\biggr)^{k-2}\biggr).
\end{align*}
Therefore, we see that the period $r(\xi_{2-k}(\mc{F}),\gamma;z)$ is the same as
\begin{align*}
\frac{(-4\pi)^{k-1}}{(k-2)!}[r^H(D^{k-1}(\mc{F}),\gamma;\bar{z})]^- =& \frac{(-4\pi)^{k-1}}{(k-2)!}\biggl\{[r(D^{k-1}(\mc{F}),\gamma;\bar{z})]^- \\
&+\delta_{\kappa,0}\overline{c_{D^{k-1}(\mc{F})}c_k}\biggl(1-c^{k-2}\chi(\gamma)\biggl(z+\frac dc\biggr)^{k-2}\biggr)\biggr\}.
\end{align*}
 This proves the first part of Theorem \ref{main3}.

In particular, if $\kappa=0,\ \lambda=1$ and $\Gamma$ is a  subgroup of finite index of $\SL_2(\ZZ)$, then by Lemma \ref{existence2}, we know that there is a weakly holomorphic modular form $H(z)\in M^!_{2-k,\chi}(\Gamma)$ with non-zero constant term $\alpha\in\CC$. Therefore,
\begin{align*}
r(D^{k-1}H,\gamma;z) &= r^H(D^{k-1}H,\gamma;z) -  \alpha c_k \biggl(1-c^{k-2}\bar{\chi}(\gamma)\biggl(z+\frac dc\biggr)^{k-2}\biggr)\\
&= -  \alpha c_k \biggl(1-c^{k-2}\bar{\chi}(\gamma)\biggl(z+\frac dc\biggr)^{k-2}\biggr),
\end{align*}
because $r^H(D^{k-1}H,\gamma;z)=0$ by {Theorem \ref{suppleperiod}}. Let $\hat{\mc{F}}(z) = \mc{F}(z) -   \frac {c_{D^{k-1}(\mc{F})}}\alpha H(z)$. Then $\xi_{2-k}(\mc{F})(z) = \xi_{2-k}(\hat{\mc{F}})(z)$, and  if we use the first part of Theorem \ref{main3}, then we have 
\begin{align*}
r(\xi_{2-k}(\hat{\mc{F}}),\gamma;z) =& r(\xi_{2-k}(\mc{F}),\gamma;z)\\
=& \frac{(-4\pi)^{k-1}}{(k-2)!}\biggl\{[r(D^{k-1}(\mc{F}),\gamma;\bar{z})]^- +\overline{c_{D^{k-1}(\mc{F})}c_k}\biggl(1-c^{k-2}\chi(\gamma)\biggl(z+\frac dc\biggr)^{k-2}\biggr)\biggr\}\\
=& \frac{(-4\pi)^{k-1}}{(k-2)!}[r(D^{k-1}(\hat{\mc{F}}),\gamma;\bar{z})]^-.
\end{align*}
The last equality comes from the computation
\begin{align*}
[r(D^{k-1}(\mc{F}),\gamma;\bar{z})]^- =& [r(D^{k-1}(\hat{\mc{F}}),\gamma;\bar{z})]^- +\frac{\overline{c_{D^{k-1}(\mc{F})}}}{\alpha}[ r(D^{k-1}H,\gamma;\bar{z})]^-\\
=& [r(D^{k-1}(\hat{\mc{F}}),\gamma;\bar{z})]^- - \overline{c_{D^{k-1}(\mc{F})}c_k}\biggl(1-c^{k-2}\chi(\gamma)\biggl(z+\frac dc\biggr)^{k-2}\biggr).
\end{align*}
This completes the proof.
\end{proof}


\subsection*{Acknowledgment}
 The authors  thank Bruce Berndt and Ken Ono for valuable comments on
an earlier version of this paper. The authors are also grateful to
anonymous referees for their careful readings and numerous suggestions
which improved the exposition of this paper a lot.

\end{document}